\documentclass[12pt]{article}

\usepackage[utf8]{inputenc}
\newcommand{\norm}[1]{\left\lVert #1 \right\rVert}

\usepackage{amsfonts}
\usepackage{amsmath}
\usepackage{amsthm}

\usepackage{mathtools}
\usepackage{tikz}
\usepackage{pgfmath}

 \usepackage{xcolor}
 
\usepackage{amssymb}

\usepackage{algorithm}
\usepackage{algorithmic}
\usepackage{todonotes}

\usepackage{multirow}

\usepackage{amsthm}
\def\R{{\mathbb R}}

\newcommand{\mPP}{\mathsf{\Phi}}
\usepackage[affil-it]{authblk}
\usepackage{hyperref}
\DeclareMathOperator*{\argmax}{argmax}
\usepackage{amsthm}
\newtheorem{theorem}{Theorem}[section]

\newtheorem{proposition}[theorem]{Proposition}
\newtheorem{remark}[theorem]{Remark}

\title{Greedy techniques for inverse problems} 
\author{L. Bruni Bruno$^{1}$, P. Massa$^{2}$, E. Perracchione$^{1}$, M. Trombini$^{1}$\\
$^1$ Dipartimento di Scienze Matematiche \lq\lq Giuseppe Luigi Lagrange\rq\rq, {Politecnico di Torino}, {Italy}\\
$^2$ University of Applied Sciences and Arts Northwest Switzerland (FHNW), School of Computer Science, Switzerland}
\date{\today}

\begin{document}
\maketitle
\textbf{Abstract.} 
Inverse imaging problems rely on limited and indirect measurements, making reconstruction highly dependent on both regularization and sample locations. We introduce a novel greedy framework for the \emph{optimal} selection of indirect measurements in the operator codomain, specifically tailored to inverse problems. Our approach employs a two-step scheme combining kernel-based interpolation and extrapolation. Within this framework, greedy schemes can be residual-based, where points are selected according to the current approximation error for a specific target function, or error-based, where points are chosen using a priori error indicators independent of the residual. For the latter, we derive explicit error bounds that quantify the propagation of approximation errors through both interpolation and extrapolation. Numerical applications to solar hard X-ray imaging demonstrate that the proposed greedy sampling strategy achieves high-quality reconstructions using only a few available measurements.

\section{Introduction}

In many applications, such as the reconstruction of medical \cite{lingala2011accelerated,ravishankar2015efficient} and astronomical \cite{prevost2022hyperspectral,piana2022hard} images, we deal with indirectly collected data. These data are captured by sensors that provide numerical information that is then used to reconstruct an image of the underlying physical phenomenon. Typically, not only data are rather limited compared to the number of pixels of the output image, but they also consist of indirect measurements. This leads to an ill-posed inverse problem (in the sense of Hadamard) and thus necessitates regularization  \cite{engl1996regularization,bertero2006inverse,Guastavino20203504}. 

In this context, the performance of image reconstruction methods depends not only on the regularization technique but also on the number and location of the available indirect measurements. The novelty of our work lies in developing algorithms for selecting an \emph{optimal} sampling in the operator codomain. We achieve this by leveraging tools from approximation theory, namely greedy methods \cite{SDW,doi:10.1137/23M1587956,WH2013,WENZEL2021105508} tailored to the \emph{inversion task}, that are capable of identifying optimal inverse sampling.

Given a set of (possibly) indirect measurements sampled at multivariate scattered points, a greedy scheme consists in considering all but a few data as test set. The data left out are used as training set to construct an initial rough model that is then refined by incrementally adding
the most relevant sample, picked among the test set, according to some iterative rule.  This rule might be based on the residuals, meaning that the greedy
scheme adds at each iteration the point at which the difference between the observed and approximated quantity is maximum, and hence the so constructed greedy
scheme is totally task dependent, or it can rely on a priori error indicators independent of the residual. In particular, residual-based methods are extremely powerful when a specific function is the target of the reconstruction, whereas error-based methods
apply to larger classes of objects, of course at the (possible) price of being
less sharp.

As we will point out, the first class of schemes can be trivially adapted to the inverse framework by applying forward and backward operators to the discrete data,  while to provide reliable  bounds  for error-based schemes, we restrict to specific imaging procedure, and precisely to interpolation and extrapolation algorithms (see e.g. \cite{duval2018solar,Perracchione2023}). This leads to the following two-step scheme:
\begin{enumerate}
    \item[(i)]  interpolating the greedily selected data;
    \item[(ii)] applying a regularized inversion method to the outcome of (i) to get an approximation of the inverse problem solution.
\end{enumerate}
A key point in (i) is the correct selection of the location of the physical data. Thus, the identification of the \emph{interpolation points} (or, more generally, supports \cite{brunisisc}) is carried via a greedy method. Despite we mainly focus on meshfree kernel-based interpolation schemes \cite{perracchione2021feature}, it is rather the basis of the projection space, in combination with the greedy method, that allows for estimates and error bounds on the reconstruction. %Last, for (ii), many methods can be applied. Nevertheless, if the regularization method is linear, an estimate on the propagation of the error can be given. 
Hence, among the variety of methods that can be applied for (ii), it is thus convenient to choose a regularization method for which an estimate on the propagation of the error can be given, so that the two contributions can be appropriately combined. This is the case, for instance, of linear methods.

The technique we develop in this paper hence lies at the (narrow) intersection between interpolation theory and regularization theory. The first part is required to set up the greedy method, while the second one is involved in the image reconstruction. The intersection is narrow in the following sense: interpolation theory usually deals with uniform approximation; as a consequence, its most renowned indicator, known as \emph{Lebesgue constant} \cite{Brutman}, is peculiar to the sup-norm. The sup-norm clearly gives a Banach structure, but is far from a Hilbert structure. 
In contrast, classical regularization theory in Banach spaces \cite{RegularizationBanach} is not compatible with the framework we adopted in this paper as it yields operators that are continuous only with respect to weak topologies. 
% In contrast, regularization theory is barely manageable in the Banach case \cite{RegularizationBanach}. 
This is what really forces us to rely on kernel methods, as they naturally bear Hilbert structures.

After investigating error estimates for the construction of greedy points, we numerically validate our theoretical claims by considering an application to solar hard X-ray imaging. 
Specifically, we consider simulated data provided by the Spectrometer/Telescope for Imaging X-rays (STIX) \cite{refId0} on board the ESA Solar Orbiter mission. Such telescope provides
observations (called visibilities) made of sampled Fourier components of the
photon flux released by solar flares. The image of the flux emitted during such explosions will be then recovered by selecting only a few measurements with our method. 

\textbf{Outiline of the paper.} The paper is organized as follows. In Section \ref{sect:scheme}, we present our scheme and review the basics of greedy methods and algorithms. In Section \ref{sect:interpolation} we describe the interpolation step and how it relates to the inverse problem given by the image reconstruction. Thus, we obtain the error estimates for the entire reconstruction process. Finally, in Section \ref{sect:numericalresult}, we present some numerical experiments that demonstrate the effectiveness of our method when applied to the astronomical framework. Conclusions are offered in Section \ref{sect:conclusion}.

\section{The scheme} \label{sect:scheme}

The abstract framework of this work consists in reconstructing a target function $x$ which is subjected to the action of a linear injective operator $A \in \mathcal L (\mathcal X, \mathcal Y) $ from two functional Banach spaces $\mathcal{X}$ and $ \mathcal Y $. 
More definitely, we focus on 
\begin{equation}
\label{eq:eq_pb_inv}
Ax=y, \quad \textrm{s.t.} \quad  y_{| \Xi} = Y
\end{equation}
being $Y\coloneqq\{y_i=y(\xi_i)\}_{i=1}^{N}$ a set of functional values sampled at a data set  $\Xi\coloneqq\{\xi_i\}_{i=1}^{N}$. 
In order to introduce the concept of greedy methods in the framework of inverse problems, we denote by 
$M_{\Xi,x} \in \mathcal X $ an approximation of the sought solution $x$. We assume that $M_{\Xi,x} $ admits point-wise evaluation.

Given $\Xi$ and $Y$, the main goal of greedy algorithms consists in selecting a suitable subset $\tilde{\Xi} \subset \Xi$ so that the {\em greedy model} $M_{\tilde \Xi,x}$ is constructed on a reduced number of data producing a surrogate model for $M_{\Xi,x}$ that approximates $x$. Such iterative algorithms belong essentially to two classes:
\begin{itemize}
    \item Residual-based greedy methods: the set $\tilde{\Xi}$ is constructed taking into account the functional values $y_i$, $i=1,\ldots,N$. 
    \item Error-based greedy methods: the set $\tilde{\Xi}$ is built independently of the functional values $y_i$, $i=1,\ldots,N$.
\end{itemize} 

Both greedy schemes, as well as their combination \cite{doi:10.1137/23M1587956,DUTTA2021110378}, are of interest. In particular, residual-based methods are target dependent, while error-based methods apply to larger classes of targets at the price of being less accurate for specific objects.

\subsection{Residual-based greedy methods}

We consider an initial (training) data set ${\tilde \Xi} = \{\xi_1\}$, without loss of generality. Then, given $Y$ and a fixed tolerance $\tau$, the residual-based greedy scheme for inverse problems is summarized in Algorithm \ref{alg:alg1}. %In this class fall the so-called $f$-greedy methods presented in \cite{SDW}.

\begin{algorithm}[H]
\small
\caption{\textbf{Pseudo-code of the residual-based greedy algorithm for inverse problems}}
\begin{algorithmic}[1]
\STATE Take an initial set of data ${\tilde \Xi} = \{ \xi_1 \}$. \vskip 0.1cm
\STATE Compute an initial approximation $M_{\tilde{\Xi},x}$. \vskip 0.1cm
\STATE While $\max_{\xi_i \in \Xi\setminus {\tilde \Xi}} \vert y(\xi_i)-(AM_{\tilde{\Xi},x})(\xi_i) \vert> \tau $: \vskip 0.1cm
    \begin{enumerate}    
    \item[(i)] Define $\xi^*= \argmax_{\xi_i \in \Xi\setminus {\tilde \Xi}} \vert y(\xi_i)-(AM_{\tilde{\Xi},x})(\xi_i) \vert$. \vskip 0.1cm
    \item[(ii)] Set ${\tilde \Xi}= {\tilde \Xi}\cup \{\xi^*\}$. \vskip 0.1cm
    \item[(iii)] Re-Compute $M_{\tilde{\Xi},x}$.
\end{enumerate}
 		\end{algorithmic}
	\label{alg:alg1}
\end{algorithm}

If no tolerance is used, i.e., all points are selected, the step at which the points are selected gives a ranking for the measurements $(\xi_i,y(\xi_i))$, $i=1,\ldots,N$. This is not frequent, as the cardinality of $\tilde{\Xi}$, namely $n$, is usually so that $n \ll N$. Residual-based greedy procedures are not significantly subjected to the initial input; in fact, if $ \Xi $ is large enough, numerical experiments show that the sensitivity to the selection of $ \xi_1 $ is negligible. It is finally plain to observe that, equivalently, one may prescribe the number of samples $ n $ in place of the tolerance $ \tau $.
We will consider the latter approach in our numerical experiments.

To conclude, we note that the residual-based greedy algorithms apply to any imaging method, as Algorithm \ref{alg:alg1} works for any  approximation of the sought solution $M_{\tilde{\Xi},x}$. This adaptivity might anyway backfire as the computation of the selected data must be done for each (possible) method and for each set of inverse measurements $y_{| \Xi}$. To avoid this drawback, we introduce the error-based greedy scheme that takes the advantage of being independent on the inverse samples. The price to pay is that the approximation $M_{\tilde{\Xi},x}$ will belong to a specific class of imaging techniques. 

\subsection{Error-based greedy methods}

Letting $\bar{y}$ be an approximation of $y$ (we will give a formal definition to this in the sequel), the basic idea of the error-based greedy is to find bounds of the general form: 
\begin{equation} \label{eq0}
 \Vert y - AM_{\tilde{\Xi},x}
\Vert_{\mathcal Y} \leq
{\mathcal{E}(\widetilde \Xi,{\bar y})} \cdot {\mathcal E (y,\bar{y},A}).
\end{equation}
In the above formula, we deliberately do not specify in which normed spaces the errors are computed, as this will be clarified later. Right now, the main message is that we have to split the error from contributions of different essence: ${\mathcal{E}}(\widetilde \Xi,{\bar{y}})$ depends on the sampling points, and on the method used to construct $\bar{y}$. In contrast, ${\mathcal{E}(y,\bar{y},A)}$ contains the remaining part of the error, which essentially depends (again) on the approximation method used to construct $\bar{y}$ and the associated residual, and on the forward operator $A$. The key fact is that the first term does not depend on $y$.

In order to further investigate the contribution of $\mathcal{E}(\widetilde \Xi, {\bar y})$, we need to represent ${\bar y}$. A possible choice consists in approximating $y$ by means of a projector, i.e., we consider 
$ \Pi y(\xi) = \sum_{i=1}^n c_i b_i(\xi) $, being $\mathcal{B} \coloneqq \textrm{span}\{b_i(\cdot)\}_{i=1}^n $ 
an appropriate set of linearly independent functions. 
The coefficients $ \{ c_i \}_{i=1}^n $ are the solution of the interpolation problem 
\begin{equation}\label{matrix_interp}
    \mPP 
    \begin{pmatrix} c_1 \\ \vdots \\ c_n \end{pmatrix}
    = \begin{pmatrix} y_1 \\ \vdots \\ y_n \end{pmatrix},
\end{equation}
where $\mPP_{ij} = b_j(\xi_i)$, $i,j=1,\ldots,n$; note that they can be uniquely determined provided that the system is non-singular, i.e. when $ \det \mPP \ne 0 $. %Note that
Since $ \bar{y}$ is obtained by projection, it is fully determined by the data sites and the basis ${\cal B}$. 
We may hence substitute the dependence on $\bar y$ in \eqref{eq0} by ${\cal B}$.
This gives 
\begin{equation} \label{eq1}
 \Vert y - AM_{\tilde{\Xi},x}
\Vert_{\mathcal Y} \leq
{\mathcal{E}(\widetilde \Xi,{\cal B})} \cdot {\mathcal E (y,\mathcal{B},A}).
\end{equation}

Under the assumption that ${\mathcal{E}(\widetilde \Xi,{\cal B})}$ admits pointwise evaluations, we may summarize the error-based greedy algorithm as follows.

\begin{algorithm}[H]
\small
\caption{\textbf{Pseudo-code of the error-based greedy algorithm for inverse problems}}
\begin{algorithmic}[1]
\STATE Take an initial set of data $\tilde{\Xi} = \{ \xi_1 \}$. \vskip 0.1cm
\STATE While $\max_{\xi_i \in \Xi\setminus \tilde{\Xi}} {\mathcal{E}}(\xi_i,\tilde{\Xi},\mathcal{B}) > \tau $: \vskip 0.1cm
    \begin{enumerate}    
    \item[(i)] Define $\xi^*={\rm argmax}_{\xi_i \in \Xi\setminus \tilde{\Xi}} {\mathcal{E}}(\xi_i,\tilde{\Xi},\mathcal{B})$. \vskip 0.1cm
    \item[(ii)] Set $\tilde{\Xi}= \tilde{\Xi}\cup \{\xi^*\}$. \vskip 0.1cm
    \item[(iii)] Re-Compute ${\mathcal{E}}(\xi_i,\tilde{\Xi},\mathcal{B})$.
\end{enumerate}
 		\end{algorithmic}
	\label{alg:alg2}
\end{algorithm}

Comparing Algorithm \ref{alg:alg2} with Algorithm \ref{alg:alg1}, one sees that the structure is similar, but in the error-based version the dependence on the residual is avoided. For, it is reasonable to expect that Algorithm \ref{alg:alg2} will perform generally worse than Algorithm \ref{alg:alg1} on a specific prefixed problem, but will give results that are broadly applicable. As for the residual-based greedy algorithm, an equivalent possibility consists in setting the cardinality $ n $ of the training set in place of the tolerance $ \tau $.

\section{The interpolation step} \label{sect:interpolation}

A formal way for computing the coefficients of the expansion of $ \Pi y = \sum_{i=1}^n c_i b_i (\cdot) $ has been given in the above section; nevertheless, such an expansion neither has a clear relationship with the problem nor is practical. Things in fact change when one selects another equivalent basis for such a space.

\subsection{The cardinal basis and the Lebesgue function}
Consider a linear finite dimensional space $\mathcal{B} \coloneqq \textrm{span}\{b_i(\cdot)\}_{i=1}^n$. If the associated matrix $\mPP_{ij}=(b_j(\xi_i))_{i,j}$, $i,j=1,\ldots, n$, is non-singular, we may write $d_{\ell k}\coloneqq(\mPP^{-1})_{\ell k}$ and we have that the functions
\begin{align*}
    \psi_\ell(\xi)\coloneqq \sum_{k=1}^n d_{\ell k}b_k(\xi),\;\; \ell = 1,\ldots, n,
\end{align*}
 are a global Lagrange (or cardinal) basis as they are so that
\begin{equation*} 
\psi_\ell(\xi_i)
= \sum_{k=1}^n d_{\ell k} b_k(\xi_i) 
= \sum_{k=1}^n (\mPP^{-1})_{\ell k} \mPP_{k i} 
= \left(\mPP^{-1}\cdot \mPP \right)_{\ell i}
=\delta_{\ell i},%\quad 1\leq i, \ell\leq n,
\end{equation*}
for $i,\ell = 1,\ldots,n$. 
Hence, 
we can represent the projector operator as
\begin{equation*}
    \Pi y(\xi) \coloneqq \sum_{i=1}^n y_i \psi_i(\xi).\;\;
\end{equation*}

Up to now, we have not specified in which normed spaces the errors in \eqref{eq1} are computed. Nevertheless, 
boundedness with respect to the sup-norm is relevant for the interpolation step. Suppose now that $ \mathcal B $ is generated by basis functions $ b_i \in L_{\infty} (\R^m) \cap C (\R^m) $. In such a case, the operator $\Pi: (L_{\infty} (\R^m) \cap C (\R^m), \Vert \cdot \Vert_\infty) \to (\mathcal{B}, \Vert \cdot \Vert_\infty) $ is bounded, and we may relate its norm with the Lebesgue constant of the space $ \mathcal{B} $ and the set $ \tilde{\Xi} $ at which samples are taken. The following is a classical result in approximation theory, which we slightly adapt to the present framework. %We recall its proof for completeness.

\begin{proposition} \label{prop:norminterpolator}
    Let $ \tilde{\Xi} = \{ \xi_1, \ldots, \xi_n \}  $ be a unisolvent set of nodes for Lagrange interpolation in $ \mathcal{B} $, then 
    \begin{equation} \label{eq:characterizationLebesgue}
        \Vert \Pi \Vert_{\mathrm{op}(\Pi)} = \sup_{\xi \in \mathbb{R}^m} \lambda (\xi) = \sup_{\xi \in \mathbb{R}^m} \sum_{i=1}^n | \psi_i (\xi) | \eqqcolon \Lambda.
    \end{equation}
\end{proposition}

\begin{proof}
    We first show that $ \Vert \Pi \Vert_{\mathrm{op}(\Pi)} \leq \Lambda $. We have that
    \begin{align*}
        \Vert \Pi \Vert_{\mathrm{op}(\Pi)} &\coloneqq \sup_{\Vert f \Vert_\infty = 1} \Vert \Pi f \Vert_\infty = \sup_{\Vert f \Vert_\infty = 1} \Vert \sum_{i=1}^n f(\xi_i) \psi_i \Vert_\infty \\
        & = \sup_{\Vert f \Vert_\infty = 1} \sup_{\xi \in \mathbb{R}^m} \left| \sum_{i=1}^n f(\xi_i) \psi_i (\xi) \right| \leq \sup_{\Vert f \Vert_\infty = 1} \Vert f \Vert_\infty \sup_{\xi \in \mathbb{R}^m}  \sum_{i=1}^n | \psi_i (\xi) | \\
        & = \sup_{\xi \in \mathbb{R}^m}  \sum_{i=1}^n | \psi_i (\xi) | = \Lambda.
    \end{align*}
    To prove that $ \Vert \Pi \Vert_{\mathrm{op}(\Pi)} = \Lambda $, it is now sufficient to exhibit a bounded function $ f \in L_\infty (\R^m ) \cap C (\R^m) $ such that $ \Vert \Pi f \Vert_{\infty} = \Lambda $. For, consider the set
    $$ F\coloneqq \{ f \in L_\infty(\mathbb{R}^m) \cap C(\mathbb{R}^m) \ \mid \ f(\xi_i) = 1 \text{ or } f(\xi_i) = -1, \ \| f \|_\infty =1 \}.$$
    Then, for every $ \xi $, there exists $ f \in F $ such that
    $$ \left| \sum_{i=1}^n f(\xi_i) \psi_i (\xi) \right| = \sum_{i=1}^n | \psi_i (\xi) | .$$
    Passing to the supremum we get
    $$ \Vert \Pi f \Vert_\infty = \sup_{\xi \in \mathbb{R}^m} \left| \sum_{i=1}^n f(\xi_i) \psi_i (\xi) \right| = \sup_{\xi \in \R^m} \sum_{i=1}^n | \psi_i (\xi) | = \Lambda, $$
    whence the claim.
\end{proof}

\subsection{Error bound of the regularized solution}

The interpolant $ \Pi y $ belongs to the finite dimensional space $ \mathcal B $. Being $ \cal B $ finite dimensional, it can be equipped with a norm that is equivalent to a Hilbert space norm. Hence, the usual regularization theory can be applied to this space and to its preimage via the injective operator $A$, which has in turn a comparable Hilbert-like structure, up to paying the equivalence constant between the two norms. Hence, let us go back to the original problem. Considering the intermediate interpolation step, estimating the error now amounts to bounding the quantity
\begin{equation*}
\Vert x^\dagger - R_\alpha \left(\Pi y \right) \Vert_{\cal X} \leq \Vert x^\dagger - R_\alpha \left( y \right) \Vert_{\cal X} + \Vert R_\alpha \left( y \right) - R_\alpha \left( \Pi y \right) \Vert_{\cal X} ~,
\end{equation*}
where $x^\dagger$ is the (unknown) solution of the inverse problem \eqref{eq:eq_pb_inv}.
The first term at the right hand side of above equation is peculiar to the regularization $ R_\alpha $ chosen for $ A $. The second one, if the regularization $ R_\alpha $ is linear, can be handled using results of the previous sections. 
From now on we stick with this hypothesis.

\begin{theorem}\label{theo:error_estimate} 
    Let $y\in L_\infty(\mathbb{R}^m)$ and let $y^{\star} \in \mathcal{B}$ be its best approximation with respect to the  $\|\cdot\|_\infty$ norm. Then
    \begin{equation*}
    \Vert R_\alpha y - R_\alpha \Pi y  \Vert_{\cal X}
    \leq \underbrace{\norm{R_\alpha}_{\mathrm{op} (R_\alpha)}\left\|y - y^{\star}\right\|_\infty}_{{\mathcal{E}(y,{\cal B},R_\alpha)}} \underbrace{\left( 1+ \sup_{\xi \in \R^m} \lambda(\xi) \right)}_{{\cal E}(\tilde{\Xi},\Xi,{\cal B})},
    \label{eq:stixIP}
    \end{equation*}
    where $\lambda(\xi)$ has been defined in \eqref{eq:characterizationLebesgue}.
% \begin{equation*}
%     \lambda(\xi)\coloneqq \sum_{j=1}^n |\psi_j(\xi)|.
% \end{equation*}
The dependence on $ \tilde \Xi $ is given in Proposition \ref{prop:norminterpolator}.
\end{theorem}
\begin{proof}
    Since $\Pi y^\star = y^\star$, using the linearity of $ R_\alpha $ we obtain
\begin{align*}
\Vert R_\alpha y - R_\alpha \Pi y \Vert_{\cal X} 
& = 
\Vert R_\alpha( y- \Pi y) \Vert_{\cal X} 
= \Vert R_\alpha( y - y^\star + y^\star - \Pi y) \Vert_{\cal X} \nonumber \\
& = \Vert R_\alpha( y - y^\star + \Pi y^\star - \Pi y) \Vert_{\cal X} \\
&= \Vert R_\alpha( \mathrm{Id} - \Pi ) (y - y^\star)  \Vert_{\cal X} \\
& \leq \Vert R_\alpha \Vert_{\mathrm{op}(R_\alpha)} \Vert y - y^\star \Vert_{\cal Y} (1 + \Vert \Pi \Vert_{\mathrm{op}(\Pi)} ).
\end{align*}
In the above equation, $ \Vert \cdot \Vert_{\mathrm{op}(R_\alpha)} $ denotes the induced norm on the space of continuous linear functionals $ \mathcal{L} ({\cal Y}, {\cal X}) $, while $ \Vert \cdot \Vert_{\mathrm{op}(\Pi)}$ is the operator norm computed with respect to $ \Vert \cdot \Vert_{\cal Y}$.
\end{proof}

\begin{remark}
The norm of the linear projection operator $ \Pi $ is also an indicator of the stability in the presence of noisy data. Indeed, if $ \Vert y - \tilde y \Vert \leq \varepsilon $, then
$$ \Vert \Pi y  - \Pi \tilde{y
}  \Vert \leq \varepsilon \Vert \Pi \Vert_{\mathrm{op}} .$$
In particular, this uncertainity propagates to the regularized solution as
$$ \Vert R_\alpha \Pi y - R_\alpha \Pi \tilde y \Vert_{\cal X} \leq \varepsilon \Vert R_\alpha \Vert_{\mathrm{op} (R_\alpha)} \Vert \Pi \Vert_{\mathrm{op} (\Pi)} .$$
\end{remark}

\begin{remark}
In the case where the regularizing operator $R_\alpha$ is non-linear, we assume that it is Frech\'et differentiable in $y$, i.e., that there exists a bounded linear operator $D_{R_\alpha}(y)$ such that
\begin{equation*}
\lim_{\Vert t \Vert_{\mathcal{Y}} \to 0} \frac{\Vert R_\alpha(t+y) - R_\alpha (y) - D_{R_\alpha}(y) \,t \Vert_{\mathcal{X}}}{\Vert t \Vert_{\mathcal{Y}}} = 0 ~.
\end{equation*}
Therefore, there exists a neighborhood $U$ of $y$ such that, for any $t\in \mathcal{Y}$ such that $t+y \in U$, $\Vert R_\alpha(t+y) - R_\alpha (y) - D_{R_\alpha}(y) t \Vert_{\mathcal{X}} \leq \Vert t \Vert_{\mathcal{Y}}$.
Substituting $t$ with $\Pi y - y $, and assuming that $\Pi y \in U$, we have 
\begin{equation*}
\Vert R_\alpha(\Pi y) - R_\alpha(y) - D_{R_\alpha}(y) (\Pi y - y)\Vert_\mathcal{X} \leq  \Vert \Pi y - y\Vert_\mathcal{Y} ~.
\end{equation*}
Therefore, by using the reverse triangular inequality, one has
\begin{equation*}
\Vert R_\alpha(\Pi y) - R_\alpha(y) \Vert_\mathcal{X} \leq \Vert D_{R_\alpha} (y) (\Pi y - y) \Vert_\mathcal{X} + \Vert \Pi y - y\Vert_\mathcal{Y} ~.
\end{equation*}
By using again the fact that $\Pi y - y = (\mathrm{Id} -\Pi) (y  - y^\star)$, we obtain 
\begin{equation*}
\Vert R_\alpha(\Pi y) - R_\alpha(y) \Vert_\mathcal{X} \leq  (1+\Vert \Pi \Vert_{\mathrm{op} (\Pi)} ) (1+\Vert D_{R_\alpha}(y) \Vert_{\mathrm{op} (D_{R_\alpha}(y))} ) \Vert y - y^\star\Vert_{\mathcal{Y}}  ~,
\end{equation*}
which allows extending the error estimate of Theorem \ref{theo:error_estimate}.

% \begin{equation}
% R_\alpha(\Pi y) - R_\alpha(y)= D R_\alpha(y) (\Pi y - y) + o(\Pi y - y) ~.
% \end{equation}
% Therefore,
% \begin{eqnarray}
% \Vert R_\alpha(y) - R_\alpha(\Pi y) \Vert_\mathcal{X} &\leq& \Vert D R_\alpha(y) (y - \Pi y) \Vert_\mathcal{X} + o(\Pi y - y) \\
% &=& \Vert D R_\alpha(y) (y - y^\star + \Pi y^\star -  \Pi y) \Vert_\mathcal{X} + o(\Pi y - y) \\
% &=& \Vert D R_\alpha(y) (\mathrm{Id} - \Pi) (y - y^\star) \Vert_\mathcal{X} + o(\Pi y - y) \\
% &\leq & \Vert D R_\alpha(y) \Vert_{\rm op} \Vert y - y^\star \Vert_{\cal Y} (1 + \Vert \Pi \Vert_{\mathrm{op}(\Pi)} )+ o(\Pi y - y)
% \end{eqnarray}
\end{remark}

From Proposition \ref{prop:norminterpolator} one immediately deduces that, with broad generality, greedy methods can be used to select nodes with a small Lebesgue constant, so that the term depending on $ \Vert \Pi \Vert_{\mathrm{op} (\Pi)} $ is minimized in the error bound. This is done in polynomial interpolation, where greedy strategies are employed for the extraction of Fekete nodes and Leja sequences \cite{Fekete}. The situation simplifies in the case of kernels, as we shall see in the forthcoming section.

\subsection{Error-based greedy: focus on kernels}

Among many interpolation spaces, we now focus on the one that naturally adapts to this context.  Precisely, in the framework of kernel interpolation \cite{Wendland05,Buhmann03,Fasshauer15}, the space $ \mathcal{B} $ can be chosen in such a way that it bears a two-fold structure. The first one is compatible with the infinity norm, when $ b_i \in L_{\infty} (\R^m) \cap C (\R^m) $ and the quantity $ \Lambda $ is finite even if the domain is unbounded. Indeed, owing to the definition given by the last equality of \eqref{eq:characterizationLebesgue}, one has
% $$ \Lambda = \sup_{\xi \in \R^m} \sum_{j=1}^n | \psi_j (\xi) | = \sup_{\xi \in\R^m} \left\vert \sum_{\ell=1}^n \Phi_{j,\ell}^{-1} b_\ell (\xi) \right\vert \leq \Vert \Phi^{-1} \Vert_\infty \cdot n \cdot \max_{\ell=1,\ldots,n} \Vert b_\ell \Vert_\infty .$$
$$ \Lambda = \sup_{\xi \in \R^m} \sum_{j=1}^n | \psi_j (\xi) | = \sup_{\xi \in\R^m} \left\vert \sum_{\ell=1}^n (\mPP^{-1})_{j \ell} b_\ell (\xi) \right\vert \leq \Vert  \mPP^{-1} \Vert_\infty \cdot n \cdot \max_{\ell=1,\ldots,n} \Vert b_\ell \Vert_\infty .$$
Notice now that $\Vert b_\ell \Vert_\infty $ is bounded when $ b_\ell \in L_\infty (\R^m) $ and that 
%$ \Vert \Phi^{-1} \Vert_\infty < \infty $ 
$\Vert \mPP^{-1}  \Vert_\infty < \infty$
since 
%$ \Phi $ 
$\mPP$ is an interpolation matrix. An example of such kernels is provided by the Matérn or the Gaussian kernels \cite{Matern}.

The second structure induced by kernels exploits an inner product, which allows us to work within the framework of Hilbert spaces and to use a preferred norm on this space, which is directly inherited from the general theory of kernel interpolation. We may see $ \mathcal{B} $ as generated by a kernel $ \kappa $ such that $ b_i (\cdot) = \kappa (\cdot, \xi_i) $. This gives to $ \mathcal B $ an explicit Hilbert structure.  Indeed, kernels  bear an inner product and associated norm  $\lVert  \cdot \lVert_{{\cal N}} = \sqrt{\langle \cdot, \cdot \rangle}$ that define their Reproducing Kernel Hilbert Spaces ${\cal N}$ (RKHS), also known as native spaces. RKHSs are so that for any $ f \in {\cal N}$ 
\begin{itemize}
    \item $\kappa(\cdot, \boldsymbol{x}) \in {\cal N}$,  
    \item $f(\boldsymbol{x})=\langle f, \kappa(\cdot,\boldsymbol{x}) \rangle$.
\end{itemize}
Since $ \mathcal B $ is finite dimensional, $ \Vert \cdot \Vert_{\mathcal N} $ is clearly equivalent to any other norm induced by $ Y $ on $ \mathcal B $. With these preliminaries, our projection assumes the form: 
\begin{equation*} 
    \Pi y  = \sum_{i=1}^n c_i \kappa(\cdot,{\xi}_i),
\end{equation*}
where the coefficients are uniquely determined by solving \eqref{matrix_interp}.

Classical pointwise error bounds for kernel-based interpolants simplify the generic computations of the previous sections. In fact, they are of the form
\begin{equation*}
    |y(\xi)-\Pi y (\xi)| \leq {\cal P} (\xi) \| y  \|_{{\cal N}}, \hskip 0.15cm y \in {{\cal N}}, 
\end{equation*}
where ${\cal P}$, known as \emph{power function}, is defined by 
\begin{align*}
    {\cal P}^2(\xi) = k(\xi,\xi)-b(\xi)^{\intercal}\mPP^{-1} b(\xi), 
\end{align*}
being $b(\xi) = (k(\xi, \xi_1), \ldots, k(\xi, \xi_n ))^{\intercal}$ and 
$$ \Vert y \Vert_{\mathcal{N}} = (y_1, \ldots, y_n) \mPP^{-1}\begin{pmatrix} y_1 \\ \vdots \\ y_n \end{pmatrix}. $$

Hence, for kernels, we can use the power function as error indicator, instead of the Lebesgue constant. This gives
\begin{equation*}
    \Vert R_\alpha y - R_\alpha \Pi y \Vert_{\cal X}
    \leq \underbrace{\norm{R_\alpha}_{\mathrm{op} (R_\alpha)}\left\|y \right\|_{\cal N}}_{{\mathcal{E}(y,{\cal B},R_\alpha)}} \underbrace{ \Vert {\cal P} \Vert_{\infty}}_{{\cal E}(\tilde{\Xi},{\cal B})},
\end{equation*}

In other words, we can apply the error-based greedy scheme by using at the power function as error estimate in Algorithm \ref{alg:alg2}. This results in an ease of implementation and a gain in computational cost compared to the non-kernel case.

\subsection{Visualization of the scheme}
Before showing our numerical results, we exemplify the greedy scheme for inversion tasks in Figure \ref{fig:scheme}. In the first figure, we select the relevant samples via either the error-based or residual-based schemes (red circles). Then, the corresponding kernel-based interpolant is constructed above such points, and evaluated in a larger number of nodes (black stars), usually a grid. Finally, a regularization is then applied on such data to reconstruct the sought solution.

\begin{figure}[H]
    \centering
		\begin{tikzpicture}[scale=0.4]
			\pgfmathsetseed{42}
			\def\s{4}
			\def\Nred{30}
			\def\Nblue{70}
			\def\margin{0.05}
			\def\dx{10}
			\def\Nredcircle{10}
			\begin{scope}[shift={(0,0)}]
				\draw (-\s,-\s) rectangle (\s,\s);
				\foreach \i in {1,...,\Nred}{
					\pgfmathsetmacro{\x}{\margin*\s + rand*(1 - 2*\margin)*\s}
					\pgfmathsetmacro{\y}{\margin*\s + rand*(1 - 2*\margin)*\s}
					\fill[blue] (\x,\y) circle (1.5pt);
				}
                \foreach \i in {1,...,\Nredcircle}{
					\pgfmathsetmacro{\x}{\margin*\s + rand*(1 - 2*\margin)*\s}
					\pgfmathsetmacro{\y}{\margin*\s + rand*(1 - 2*\margin)*\s} 
			      \draw[red] (\x,\y) circle[radius=1.5mm];
                 \fill[blue] (\x,\y) circle (1.5pt);
				}
			\end{scope}
			\draw[->, thick, bend left=20] (\s+0.05,0) to (\dx - \s-0.05,0);
			\begin{scope}[shift={(\dx,0)}]
				\draw (-\s,-\s) rectangle (\s,\s);
                \foreach \i in {1,...,\Nredcircle}{
					\pgfmathsetmacro{\x}{\margin*\s + rand*(1 - 2*\margin)*\s}
					\pgfmathsetmacro{\y}{\margin*\s + rand*(1 - 2*\margin)*\s} 
			      \fill[red] (\x,\y) circle[radius=1.5mm];
				}
				\foreach \i in {1,...,\Nblue}{
					\pgfmathsetmacro{\x}{\margin*\s + rand*(1 - 2*\margin)*\s}
					\pgfmathsetmacro{\y}{\margin*\s + rand*(1 - 2*\margin)*\s}
					\node at (\x,\y) {*};
				}
			\end{scope}
			\draw[->, thick, bend left=20] (\dx + \s+0.05, 0) to (\dx*2 - \s-0.05, 0);
			\begin{scope}[shift={(\dx*2,0)}]
				\draw[step=0.4,thin,gray!60] (-\s,-\s) grid (\s,\s);
				\draw[thick] (-\s,-\s) rectangle (\s,\s);
				\shade[inner color=blue!70, outer color=white] (1.5,2) circle (1.2);
			\end{scope}
			
		\end{tikzpicture}
    \caption{The scheme: combination of kernel interpolation and regularization. In the first image, the most important samples are selected (first image, red circles). This gives an interpolant, which is then evaluated at more points (second image, black stars). A regularization is then applied on such data to reconstruct the desired image (third image).}
    \label{fig:scheme}
\end{figure}
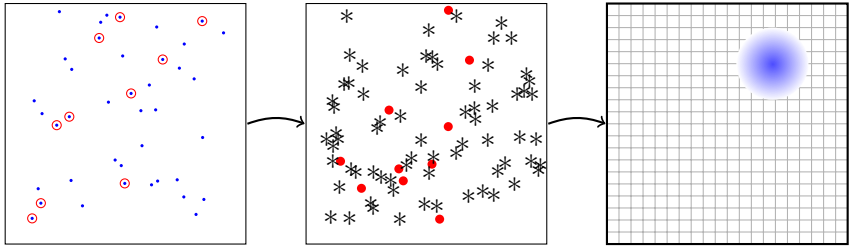

\section{Numerical Experiments: Applications to Solar Imaging} \label{sect:numericalresult}

We validate the theoretical results on an inverse imaging problem using synthetic data generated for the STIX instrument. STIX is a telescope that measures the hard X-ray radiation emitted by solar flares, i.e., sudden bursts of radiation occurring in the Sun atmosphere. The data provided by STIX, known as visibilities $Y = \{ y(\xi_i) \}_{i=1}^N$, consist of samples of the two-dimensional Fourier transform of the incident X-ray photon flux \cite{massa2023stix}.

Thus, the STIX imaging problem, which consists in recovering the spatial distribution of hard X-ray sources from the corresponding Fourier measurements, can be formulated as the ill-posed inverse problem in \eqref{eq:eq_pb_inv}, where the forward operator $A$ is replaced by the Fourier transform $\mathcal{F} \colon L_1(\mathbb{R}^2) \to L_\infty(\mathbb{R}^2) \cap C(\mathbb{R}^2)$. Several reconstruction techniques have been proposed for this task (see, e.g., \cite{piana2022hard}), including interpolation and extrapolation approaches (see, e.g., \cite{perracchione2021feature, Perracchione_2021}), commonly referred to as uv\_smooth in the hard X-ray imaging literature. In particular, its current implementation employs variably scaled kernels \cite{Bozzini1}, which are known to potentially improve reconstructions for oscillating target structures \cite{Perracchionepolito}, while regularization is performed using the Landweber scheme, refer e.g. to \cite{piana1997projected}.

For the experiments presented here, we consider a set of $N = 400$ Fourier frequencies, given by Fibonacci nodes $\Xi = \{\xi_i\}_{i=1}^N$ (left panel of Figure \ref{fig:sampling_points_P}). Fibonacci nodes have been shown to be an effective choice for interpolation-based techniques in the context of hard X-ray solar flare imaging \cite{perracchione2021feature}. Further, as discussed in Section \ref{sect:scheme}, a larger sampling set is required to properly validate our proposed greedy strategies for the first time.

\begin{figure}[H]
    \centering
    % --- sampling points ---
    % \begin{subfigure}[b]{0.45\textwidth}
    %     \includegraphics[width=\textwidth]{P_Fibonacci_Sampling_NoTitle_No_uv.png}
    %     \caption{All sampling points}
    % \end{subfigure}
    % \hfill
    % \begin{subfigure}[b]{0.45\textwidth}
    %     \includegraphics[width=\textwidth]{P_Fibonacci_Sampling_Selected_NoTitle_No_uv.png}
    %     \caption{Selected points}
    % \end{subfigure}
    
    \includegraphics[width=0.45\textwidth]{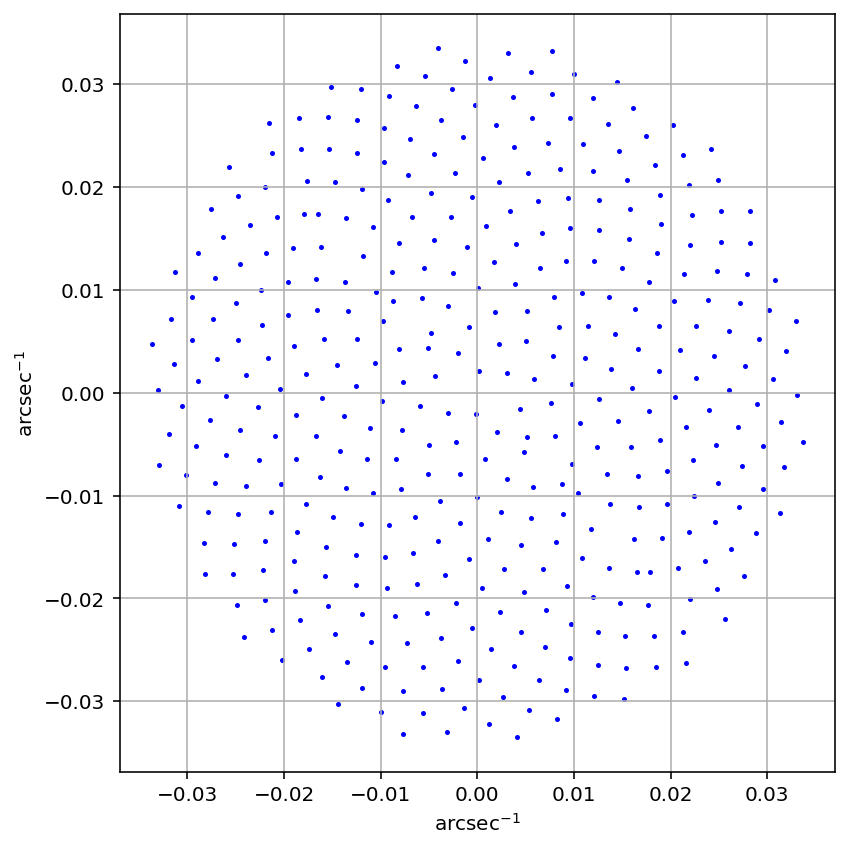}
    \includegraphics[width=0.45\textwidth]{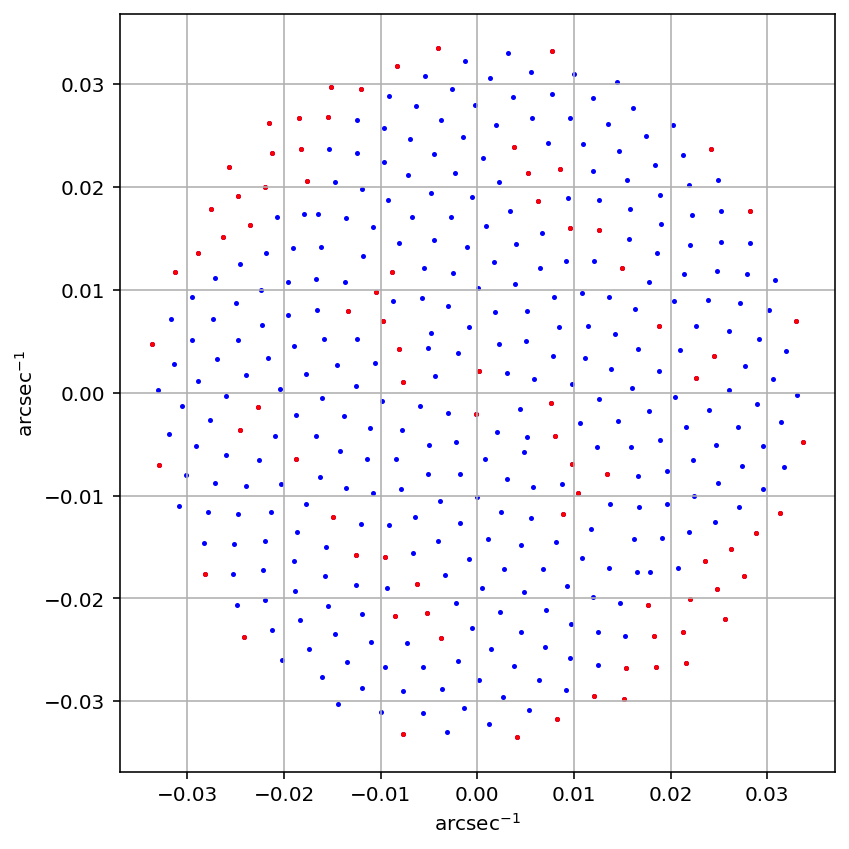}
 
    \caption{The full set of Fibonacci nodes used in our experiments is shown in blue in both the left and right panels, whereas the subset of frequencies selected by the error-based greedy strategy is highlighted in red in the right panel.}
    \label{fig:sampling_points_P}
\end{figure}

Given the limited information provided by hard X-ray telescopes, the reconstructed images of solar flare sources are relatively simple and can typically be represented using either two-dimensional Gaussian functions or bent Gaussian functions (referred to as loops). Accordingly, in the experiments below, we simulate STIX data corresponding to three prototypical solar flare morphologies: a single Gaussian source (hereafter \emph{single}), a double Gaussian configuration (hereafter \emph{double}), and a loop-shaped structure (hereafter \emph{loop}). The corresponding ground-truth images are shown in Figure \ref{fig:ground_truth}.

As far as the imaging methods are concerned, we utilize methods already tested in the hard X-ray imaging literature. We compare the results of the greedy strategies of:
\begin{itemize}
    \item Our interpolation/extrapolation procedure  hereafter referred to as \textbf{uv\_smooth} in accordance with astrophysical terminology;
    \item The maximum entropy method  \textbf{MEM\_GE} \cite{Massa_2020};
    \item The classical  \textbf{Clean} de-convolution algorithm which is inherited from the field of solar radio interferometry \cite{Clean}. 
\end{itemize}

\begin{figure}[H]
    \centering
    % --- ground truth ---
    \includegraphics[width=\textwidth]{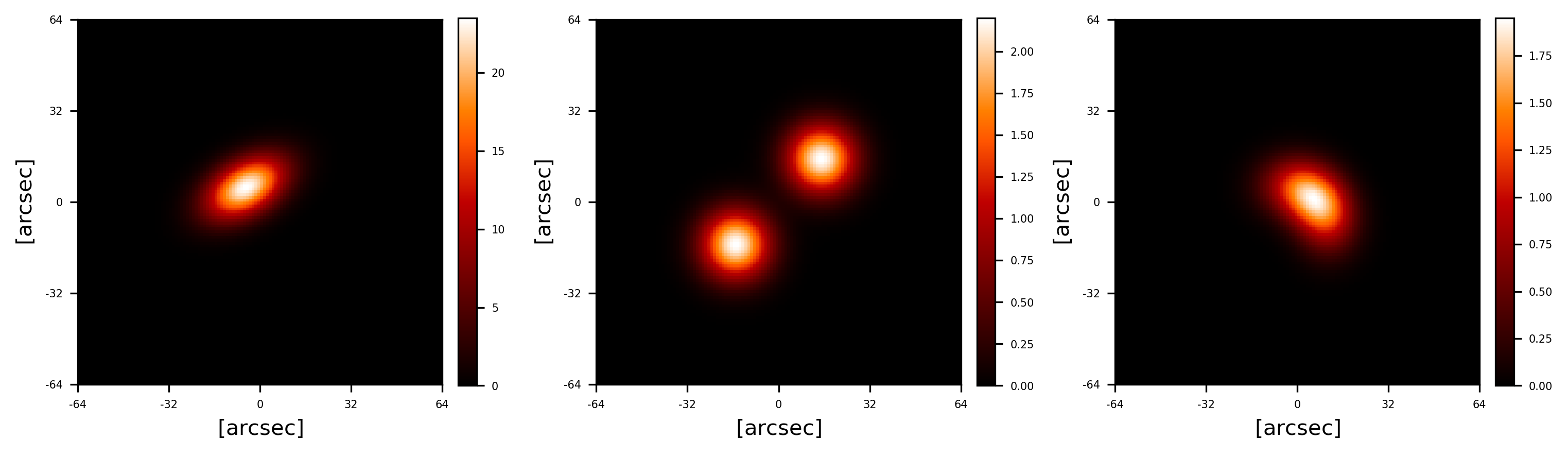}
    % \caption{Ground truth images corresponding to the different test cases: \textit{One Source}, \textit{Two Sources} and \textit{Loop}.}
    \caption{Ground truth images of the simulated flare morphologies: \emph{single}, \emph{double}, and \emph{loop} (from left to right, respectively).}
    \label{fig:ground_truth}
\end{figure}

Our numerical experiments can be reproduced using the code available at 
\begin{center}
\url{https://github.com/MatteTro/GreedyIP}
\end{center}
\subsection{Results}

In the case of the error-based greedy approach, we apply Algorithm \ref{alg:alg2} to the set of Fibonacci nodes and terminate the iterative procedure once a total of $n = 80$ frequencies has been selected. This technique is random and depends on the initial frequency $\xi_1$. Nevertheless, the error-based greedy approach proves to be stable with respect to the initialization and consistently selects a subset of frequencies that is uniformly distributed across the original set of Fibonacci nodes (see the right panel of Figure \ref{fig:sampling_points_P}). Such a distribution prevents regions of the Fourier domain from being overlooked and ensures that the different components or variations present in the original data remain represented within the subsample. This behavior is not unexpected, as the subset is chosen independently of any specific flare data and therefore corresponds to a configuration of frequencies that yields near-optimal performance in a general setting for our interpolation/extrapolation procedure.

\begin{figure}

\includegraphics[width=\textwidth]{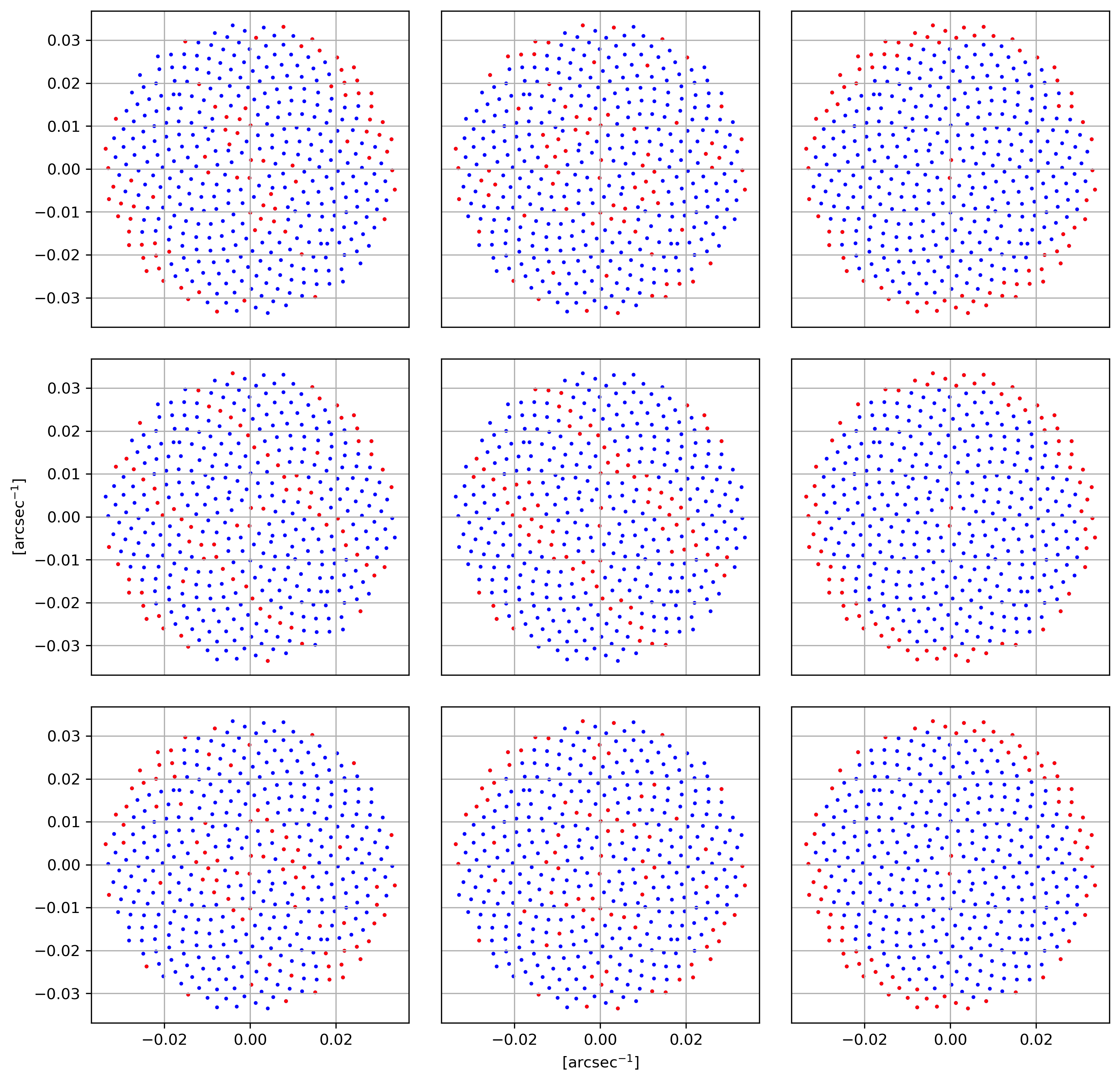}
\caption{Sampling points selected by the \textit{residual-based} greedy strategy for the \emph{single}, \emph{double}, and \emph{loop} configuration (top, middle, and bottom row, respectively).
From left to right, the panels correspond to the sampling subsets chosen when using 
the uv\_smooth, MEM\_GE, and Clean.}
\label{fig:fig:sampling_points_F}

\end{figure}

\begin{figure}[!t]
    \includegraphics[width=\textwidth]{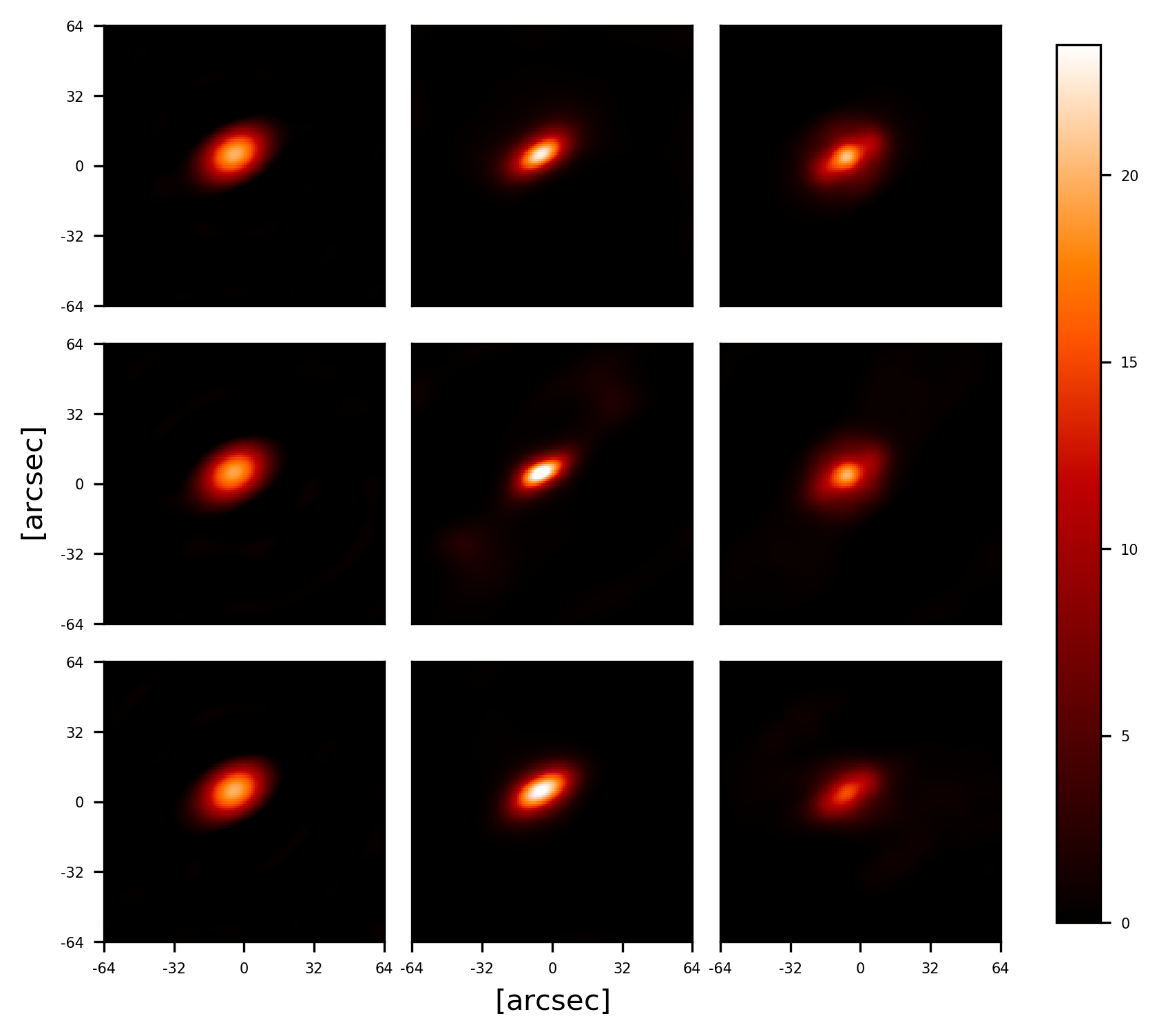}
    \caption{Reconstructed images in the case of the \emph{single} configuration. 
    Columns correspond to the three reconstruction methods: uv\_smooth, MEM\_GE, and Clean (from left to right, respectively). 
    Rows show the reconstruction obtained using (top) all sampling points,  (middle) the subset selected by the error-based greedy method, and (bottom) the subset selected by the residual-based greedy method.}
    \label{fig:Reconstructions_OneSource}
\end{figure}

\begin{figure}[!t]
    \includegraphics[width=\textwidth]{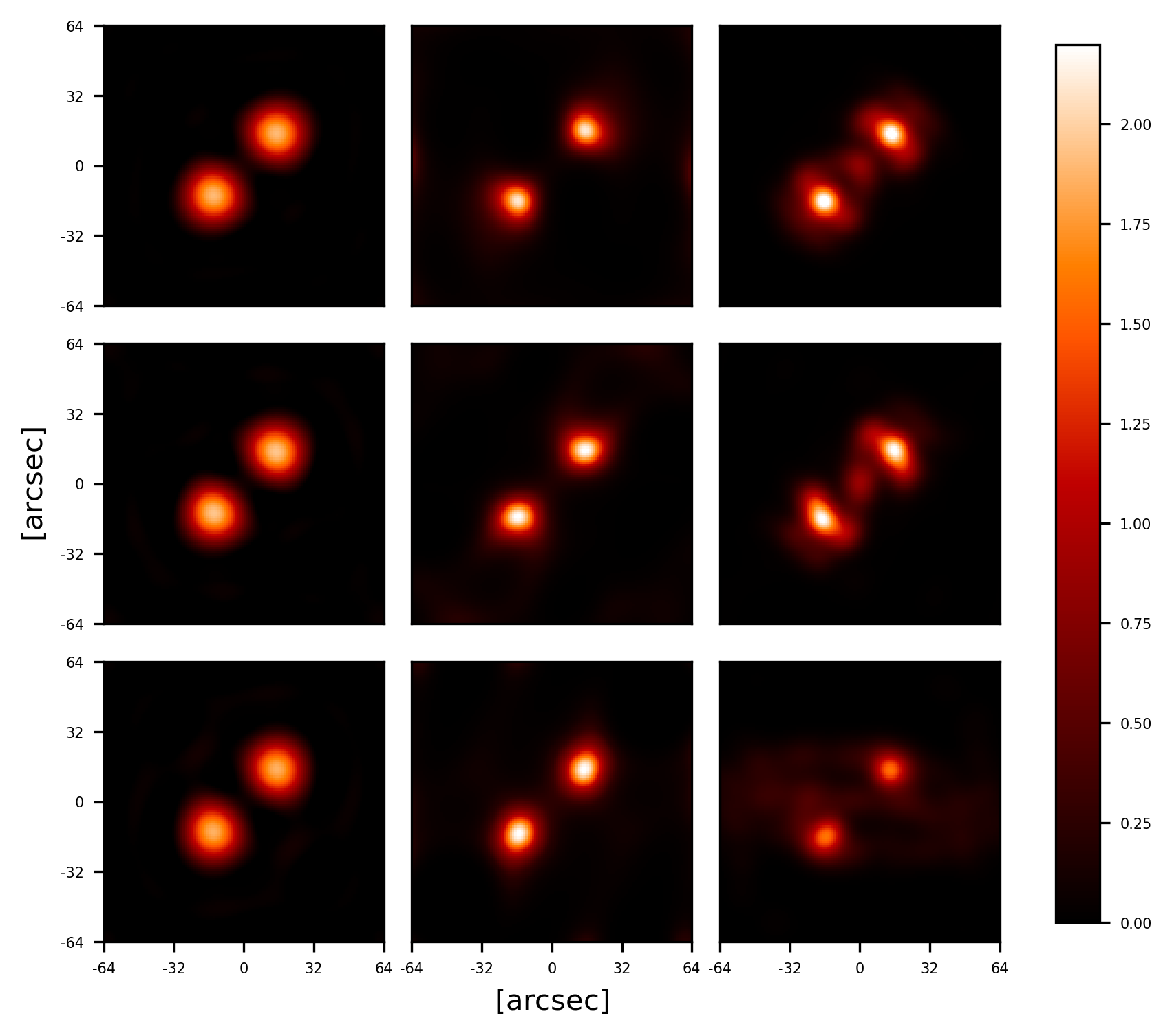}
    \caption{Reconstructed images in the case of the \emph{double} configuration. 
    Columns correspond to the three reconstruction methods: uv\_smooth, MEM\_GE, and Clean (from left to right, respectively). 
    Rows show the reconstruction obtained using (top) all sampling points,  (middle) the subset selected by the error-based greedy method, and (bottom) the subset selected by the residual-based greedy method.}
    \label{fig:Reconstructions_TwoSources}
\end{figure}

\begin{figure}[!t]
    \includegraphics[width=\textwidth]{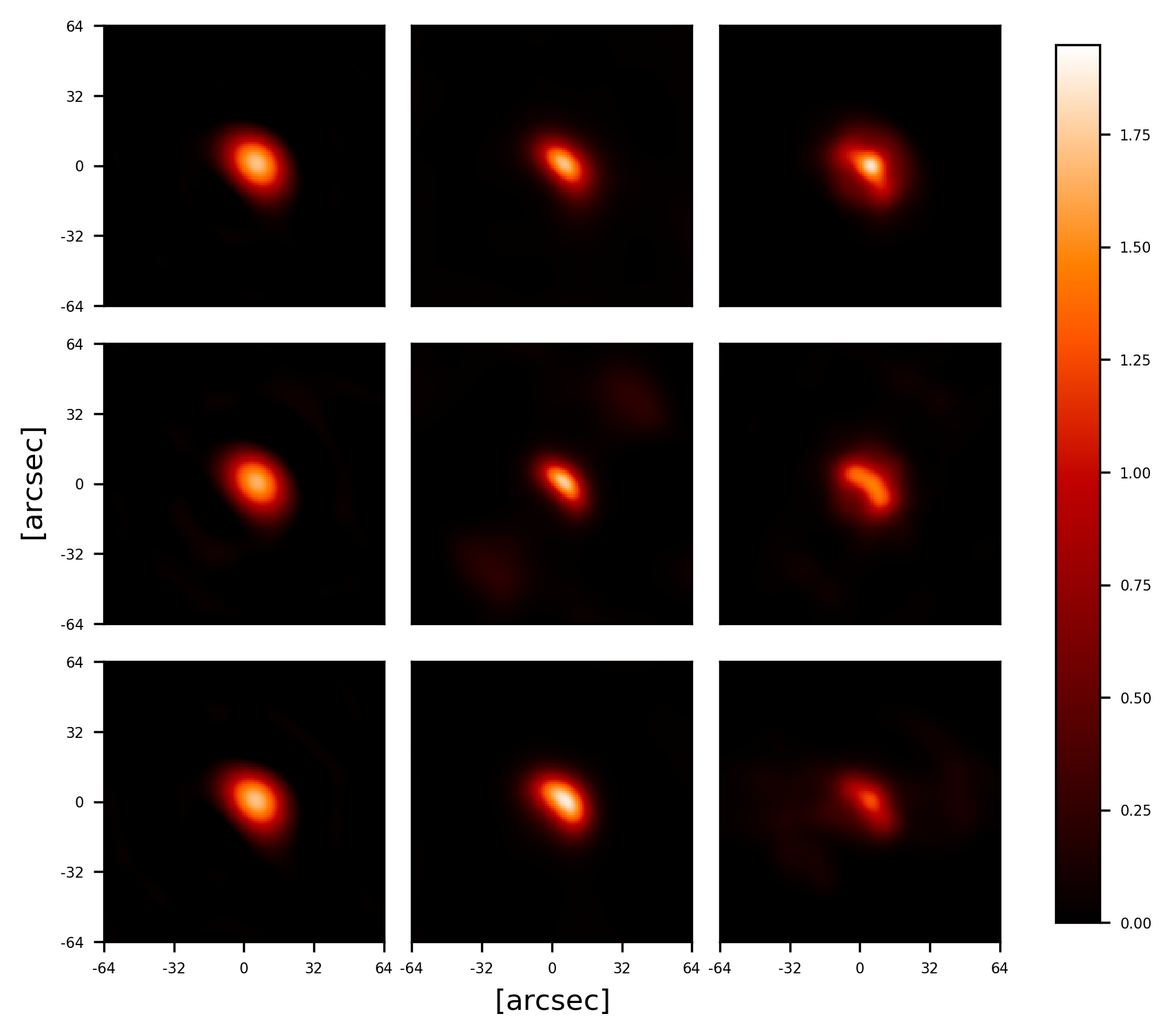}
    \caption{Reconstructed images in the case of the \emph{loop} configuration. 
    Columns correspond to the three reconstruction methods: uv\_smooth, MEM\_GE, and Clean (from left to right, respectively). 
    Rows show the reconstruction obtained using (top) all sampling points,  (middle) the subset selected by the error-based greedy method, and (bottom) the subset selected by the residual-based greedy method.}
    \label{fig:Reconstructions_Loops}
\end{figure}

Differently, the residual-based greedy technique yields an optimal set of frequencies tailored to both the specific image reconstruction algorithm and the particular flare data under consideration. We applied Algorithm \ref{alg:alg1} to simulated data corresponding to the \emph{single}, \emph{double}, and \emph{loop} configurations, and to each image reconstruction algorithm under consideration.

The resulting sets of selected frequencies are shown in Figure \ref{fig:fig:sampling_points_F}. 
It can be noted that the selected frequencies are different for every method and every ground truth configuration.
We also note a consistency between the frequencies selected by uv\_smooth and MEM\_GE, while Clean produces a more scattered distribution, with a stronger emphasis on peripheral regions.

In Figures \ref{fig:Reconstructions_OneSource}, \ref{fig:Reconstructions_TwoSources}, and \ref{fig:Reconstructions_Loops}, we compare the reconstructions obtained using all Fibonacci points with those obtained using the subsets of frequencies selected by the error-based and residual-based greedy approaches, for the single, double, and loop configurations, respectively. The three columns in each figure display the reconstructions produced by uv\_smooth, MEM\_GE, and Clean.

Reconstructed images highlight the influence of the sampling strategy. The residual-based greedy selection yields more targeted reconstructions that tend to preserve the main source structure more effectively and produce fewer artifacts. In contrast, the error-based selection exhibits a slight decrease in reconstruction quality compared to the residual-based approach. However, its quality remains comparable to that obtained without any greedy selection. Naturally, this still represents a significant improvement, as the number of considered frequencies is substantially reduced. Moreover, the error-based greedy algorithm can be executed once and for all, providing a clear advantage in terms of computational cost.

As for uv\_smooth, the main differences appear in the artifacts along the outer regions of the reconstructed flare: these are visible for both greedy strategies but remain relatively mild and do not compromise the overall structure.
Further, as already noted in the literature, the \emph{double} configuration remains a critical case for this method, as it tends to merge the two sources.
MEM\_GE maintains good performance under both sampling strategies, although error-based selection introduces slightly more distortion and results in weaker separation between the two sources. 

While MEM\_GE is known to be a robust technique, the performance of Clean may be affected by a suboptimal choice of the width of the Gaussian beam used in the final step of the algorithm to convolve the identified point sources. There is no mathematically rigorous rule to justify this choice, and heuristic considerations are typically adopted. Clean, which already provides inaccurate reconstructions when all Fibonacci nodes are considered, shows a further decrease in performance when frequencies are removed, particularly in the case of the error-based greedy approach.

Finally, the error-based greedy strategy represents a good choice for uv\_smooth, whereas both MEM\_GE and Clean may be negatively impacted by a configuration tailored solely for the interpolation/extrapolation method.

To provide a quantitative assessment of the reliability of the proposed greedy strategies, we report in Table \ref{tab:metrics} the values of several metrics, namely the $\chi^2$, Root Mean Squared Error (RMSE), and Mean Relative Error (MRE):
\begin{eqnarray*}
 \chi^2 &\coloneqq& \dfrac{1}{N} \sum_{i=1}^N \dfrac{\vert y(\xi_i) - (\mathcal{F} M_{\tilde{\Xi},x})(\xi_i) \vert}{\sigma_i^2}~, \\
\mathrm{RMSE} &\coloneqq& \sqrt{\frac{1}{N} \sum_{i=1}^N \vert y(\xi_i) - (\mathcal{F} M_{\tilde{\Xi},x})(\xi_i) \vert^2} ~, \\
\mathrm{MRE} & \coloneqq & \dfrac{1}{N} \sum_{i=1}^N \dfrac{\vert y(\xi_i) - (\mathcal{F} M_{\tilde{\Xi},x})(\xi_i) \vert}{\vert y(\xi_i) \vert} ~,
\end{eqnarray*}
where $\tilde{\Xi}$ can be either the complete set of Fibonacci frequencies or the subset selected by greedy algorithms, and $\sigma_i$ is the uncertainty related to the absolute value of the $i$-th visibility.
The three metrics compare the simulated visibilities corresponding to \emph{all} Fibonacci nodes with those predicted from the reconstructions obtained with the different methods. 

As for uv\_smooth, considering the frequencies selected by the residual-based greedy approach leads to a systematic improvement in the metric values, with the only exception being the $\chi^2$ metric in the case of the \emph{double} configuration. This result highlights the advantage of adopting a frequency-selection strategy tailored to this imaging method and the specific data at hand. 
Notably, using the frequencies selected by the error-based greedy approach yields results comparable to (and sometimes even better than) those obtained with all Fibonacci points, despite only one fifth of the data being considered. The improvement (or even the stabilization) with respect to the $ \chi^2 $ metric is remarkable, as none of the greedy methods is directly targeted at it. This can also be done, but carries a large computational cost.

MEM\_GE and Clean provide suboptimal results for the \emph{single} configuration, particularly when using all Fibonacci frequencies or the error-based greedy frequencies. The performance of the maximum entropy method improves significantly when residual-based frequencies are employed; however, the same improvement is not observed for Clean, which may be affected by a suboptimal choice of Gaussian beam width, as discussed above.

Similar observations apply to Clean in the \emph{double} and \emph{loop} configurations, although these cases are less critical than the \emph{single} configuration. In the same configurations, MEM\_GE suffers when frequencies selected by the error-based greedy approach are used, showing almost systematically lower accuracy than both uv\_smooth and Clean. However, the maximum entropy method again shows notable improvement when residual-based greedy frequencies are adopted, providing results comparable to uv\_smooth.

\begin{table}[!t]
\centering
\caption{Data-fitting metrics for different sampling strategies across three test cases.}
\resizebox{\textwidth}{!}{%
\begin{tabular}{ccc}
% ================= Prima tabella =================
\begin{tabular}{lccc}
\multicolumn{4}{c}{\emph{Single} configuration} \\ \hline
\textbf{Sampling} & $\chi^2$ & RMSE & MRE \\
\hline
\textbf{All points} & & & \\
uv\_smooth & 1.02 & 0.10 & 0.09 \\
MEM\_GE    & 7.49 & 0.22 & 0.21 \\
Clean    & 6.71 & 0.30 & 0.27 \\
\hline
\textbf{Error-based} & & & \\
uv\_smooth & 1.39 & 0.08 & 0.08 \\
MEM\_GE    & 15.69 & 0.25 & 0.20 \\
Clean    & 9.09 & 0.32 & 0.29 \\
\hline
\textbf{Residual-based} & & & \\
uv\_smooth & 0.82 & 0.08 & 0.07 \\
MEM\_GE    & 0.42 & 0.05 & 0.04 \\
Clean    & 20.34 & 0.40 & 0.38 \\
\hline
\end{tabular}
&
% ================= Seconda tabella =================
\begin{tabular}{lccc}
\multicolumn{4}{c}{\emph{Double} configuration} \\ \hline
\textbf{Sampling} & $\chi^2$ & RMSE & MRE \\
\hline
\textbf{All points} & & & \\
uv\_smooth & 0.17 & 0.48 & 0.21 \\
MEM\_GE    & 2.04 & 1.10 & 0.54 \\
Clean    & 0.67 & 0.31 & 0.28 \\
\hline
\textbf{Error-based} & & & \\
uv\_smooth & 0.27 & 0.79 & 0.21 \\
MEM\_GE    & 1.09 & 0.40 & 0.25 \\
Clean    & 0.60 & 0.35 & 0.58 \\
\hline
\textbf{Residual-based} & & & \\
uv\_smooth & 0.27 & 0.44 & 0.21 \\
MEM\_GE    & 0.77 & 0.42 & 0.25 \\
Clean    & 4.01 & 0.80 & 0.58 \\
\hline
\end{tabular}
&
% ================= Terza tabella =================
\begin{tabular}{lccc}
\multicolumn{4}{c}{\emph{Loop} configuration} \\ \hline
\textbf{Sampling} & $\chi^2$ & RMSE & MRE \\
\hline
\textbf{All points} & & & \\
uv\_smooth & 0.08 & 0.08 & 0.07 \\
MEM\_GE    & 0.93 & 0.25 & 0.24 \\
Clean    & 0.55 & 0.30 & 0.27 \\
\hline
\textbf{Error-based} & & & \\
uv\_smooth & 0.21 & 0.11 & 0.10 \\
MEM\_GE    & 2.45 & 0.34 & 0.29 \\
Clean    & 0.76 & 0.33 & 0.30 \\
\hline
\textbf{Residual-based} & & & \\
uv\_smooth & 0.07 & 0.07 & 0.06 \\
MEM\_GE    & 0.14 & 0.10 & 0.09 \\
Clean    & 2.92 & 0.46 & 0.45 \\
\hline
\end{tabular}
\end{tabular}
\label{tab:metrics}
}
\end{table}

\section{Conclusions and future work} \label{sect:conclusion}
We presented a greedy framework for inverse problems, emphasizing two sampling strategies: residual-based and error-based. Residual-based selection efficiently reduces the discrepancy between predicted and observed data, while error-based selection allows direct control over reconstruction error with explicit bounds. Numerical results on STIX solar imaging show that both strategies achieve high-quality reconstructions with a limited number of measurements. The error-based approach, in particular, provides a theoretically grounded criterion for optimal measurement selection. This highlights the potential of greedy methods for efficient and reliable sparse reconstructions.

Future work consists in taking into account possible noise on the measurements and studying how the uncertainties propagate on the sought solution.

\end{document}